\documentclass[12pt]{article}

\usepackage{float}
\usepackage[colorlinks=true,
linkcolor=blue,
urlcolor=blue,
citecolor=blue]{hyperref}
 \usepackage{enumerate}
 \usepackage{amsthm,amsmath,amssymb}
\usepackage[margin = 2.8cm]{geometry}
\usepackage{amsthm,amsmath,amssymb}
\usepackage{graphicx}
\usepackage{float}
\usepackage{hyperref}
\usepackage{mathtools}
\usepackage{cite}
 \usepackage{enumerate}

\usepackage[toc,page]{appendix}
\usepackage{lscape}
\usepackage{multirow}
\usepackage{adjustbox,expl3,etoolbox}
\usepackage{color} 
\usepackage[utf8]{inputenc}

\newtheorem{thm}{Theorem}[section]
\newtheorem{lem}[thm]{Lemma}
\theoremstyle{definition}

\newtheorem{quest}[thm]{Question}

\DeclareMathOperator\Cay{Cay}

\DeclareMathOperator\agl{AGL}
\DeclareMathOperator\gl{GL}

\title{\bf On non-normal subgroup perfect codes}

\author{
Angelot Behajaina \\
\small Laboratoire de Mathématiques Nicolas Oresme, \\[-0.8ex]
\small Université de Caen Normandie\\[-0.8ex] 
\small BP 5186, 14032 Caen Cedex,  France\\
\small\tt angelot.behajaina@unicaen.fr \vspace*{0.5cm}\\	
Roghayeh Maleki \\
\small Department of Mathematics and Statistics \\[-0.8ex]
\small University of Regina \\[-0.8ex] 
\small 3737 Wacana Parkway, Regina, Saskatchewan, Canada \\
\small\tt {rmaleki}@uregina.ca 
\vspace*{0.5cm}\\
 Andriaherimanana Sarobidy Razafimahatratra\\
\small Department of Mathematics and Statistics \\[-0.8ex]
\small University of Regina \\[-0.8ex] 
\small 3737 Wacana Parkway, Regina, Saskatchewan, Canada \\
\small\tt sarobidy@phystech.edu \\
}

\begin{document}	

\maketitle

\begin{abstract}
  Let $X = (V,E)$ be a graph. A subset $C \subseteq V(X)$ is a \emph{perfect code} of $X$ if $C$ is a coclique of $X$ with the property that any vertex in $V(X)\setminus C$ is adjacent to exactly one vertex in $C$. Given a finite group $G$ with identity element $e$ and $H\leq G$, $H$ is a \emph{subgroup perfect code} of $G$ if there exists an inverse-closed subset $S \subseteq G\setminus \{e\}$ such that $H$ is a perfect code of the Cayley graph $\Cay(G,S)$ of $G$ with connection set $S$. In this short note, we give an infinite family of finite groups $G$ admitting a non-normal subgroup perfect code $H$ such that there exists $ g\in G$ with $g^2\in H$ but $(gh)^2 \neq e$, for all $h \in H$; thus, answering a question raised by Wang, Xia, and Zhou in [Perfect sets in Cayley graphs. {\it arXiv preprint} arXiv:2006.05100, 2020].
\end{abstract}

\section{Introduction} 

The notion of perfect codes is fundamental to coding theory.
In 1973, Biggs \cite{biggs1973perfect} extended this concept for distance-transitive graphs, which led to various generalizations for association schemes and simple graphs \cite{bannai1977perfect,delsarte1973algebraic,hammond1975perfect,smith1980perfect}. The generalization of perfect codes for simple graphs is of particular interest to us. Given a graph $X = (V,E)$ and $t\in \mathbb{N}$, a subset $C \subseteq V(X)$ is a \emph{perfect $t$-code} if for every vertex $x\in V(X)$, there exists exactly one vertex $c\in C$ which is at distance at most $t$ from the vertex $x$. In particular, $C$ is a \emph{coclique} or an \emph{independent set} of the graph $X$.
A perfect $1$-code of $X$ is called a \emph{perfect code}.

One can also extend the concept of perfect codes for groups. Given a finite group $G$ with identity element $e$ and a subset $S \subseteq G\setminus \{e\}$ which is inverse-closed (i.e., if $x\in S$ then $x^{-1} \in S$), the Cayley graph $\Cay(G,S)$ is the graph whose vertex set is the group $G$ and whose edge set consists of pairs $(g,h) \in G\times G$ such that $hg^{-1} \in S$. As $S$ is inverse-closed, the graph $\Cay(G,S)$ is a simple graph. A subset $C$ of $G$ is a perfect code of $G$ if $C$ is a perfect code of a Cayley graph of $G$. In other words, there exists an inverse-closed subset $S$ of $G\setminus \{e\}$ such that $C$ is a perfect code of $\Cay(G,S)$. If $H\leq G$ is a perfect code of $G$, then we say that $H$ is a \emph{subgroup perfect code} of $G$. 

Perfect codes for groups have been well-studied in the past decade \cite{feng2017perfect,huang2018perfect,lee2001independent,ma2020subgroup,zhou2019cyclotomic}. For instance, Huang, Xia, and Zhou \cite{huang2018perfect} gave a necessary and sufficient condition for a normal subgroup to be a perfect code.
\begin{thm}[\cite{huang2018perfect}]
	Let $G$ be a group with identity element $e$, and let $H \triangleleft G$. Then, $H$ is a perfect code of $G$ if and only if the following formula, $\Phi(G,H)$, holds:
	\begin{align*}
	\forall g\in G\ (g^2\in H) \Rightarrow \exists h\in H, (gh)^2 = e.
	\end{align*}
	\label{lem:char1}
\end{thm}

Throughout this paper, we use $G$ to denote a finite group and $e$ to denote the identity of $G$.
In \cite{wang2020perfect}, Wang, Xia, and Zhou asked the following question. 
\begin{quest}
	Does Theorem~\ref{lem:char1} still hold when $H$ is a non-normal subgroup of $G$?\label{quest:main}
\end{quest}

In this short note, we show that $\Phi(G,H)$ is no longer a necessary condition when $H$ is not normal. Consequently, we give a negative answer to Question~\ref{quest:main}. To do this,  we provide an infinite family of examples.
Fix a positive integer $n \geq 1$. Set $q=2^n$ and let $\alpha$ be a primitive element of the quadratic extension $\mathbb{F}_{q^2}/\mathbb{F}_q$. We have $\mathbb{F}_{q^2}=\mathbb{F}_q \oplus \mathbb{F}_q \alpha$.
Consider the affine group $$\mathrm{AGL}(2,q^2) := \left\{ (a,A) \mid a\in \mathbb{F}_{q^2}^2 \mbox{ and } A \in \gl_2(\mathbb{F}_{q^{2}}) \right\},$$ 
with multiplication $(a,A)(b,B) = (a+Ab,AB)$, for any $(a,A),\ (b,B) \in \agl(2,q^2)$. 

For any $T \subseteq \mathbb{F}_{q^2}$, we let $\begin{pmatrix}T \\ T \end{pmatrix}$ be the set of all vectors of $\mathbb{F}_{q^2}^2$ with entries in $T$.
Let $H_q$ be the subgroup of $\agl(2,q^2)$ given by 
\begin{align*}
H_q:=\left \lbrace (b,\mathrm{I}_2) \in \agl(2,q^2) \ \middle| \ b \in \begin{pmatrix}\mathbb{F}_q \\ \mathbb{F}_q \end{pmatrix}\right \rbrace,
\end{align*}
where $\mathrm{I}_2$ is the identity matrix. Our main result is stated as follows.
\begin{thm}
	The subgroup $H_q$ is a non-normal subgroup of $\agl(2,q^2)$ which is a perfect code but $\Phi\left(\agl(2,q^2),H_q \right)$ does not hold.\label{thm:main-result}
\end{thm}

\section{Proof of Theorem~\ref{thm:main-result}}\label{section:counterexamples}

%%%%%%%%%%%%%%%%%%%%%%%%%%%%%%%%%%%%%%%%%%%%%%%%%%
\subsection{Main lemmas}
We recall that when $G$ is a group and $H$ is a subgroup of $G$, then a subset $S \subset G$ is a \emph{left transversal} of $H$ in $G$ if for any $g\in G$, we have $|gH \cap S| = 1$. A few general characterizations of subgroup perfect codes are given next.

\begin{lem}[\cite{ma2019subgroup}]
	Let $G$ be a group and $H \leq G$. Then, $H$ is a perfect code of $G$ if and only if $H$ has an inverse-closed left transversal.
	\label{carma2019}
\end{lem}

\begin{lem}\cite[Corollary~3.3]{zhang2021subgroup}
	Let $G$ be a group and let $H\leq G$ be a $2$-group. Then, $H$ is a perfect code of $G$ if and only if $\Phi(N_{G}(H),H)$ holds, where $N_G(H)$ is the normalizer of $H$ in $G$.\label{lem:char2}
\end{lem}
We note that a much stronger statement than Lemma~\ref{lem:char2} was first proved in \cite[Theorem~3.1, Theorem~3.2]{zhang2020subgroup}, however the proof contained an error. This was subsequently corrected in \cite{zhang2021subgroup}.
%%%%%%%%%%%%%%%%%%%%%%%%%%%%%%%%%%%%%%%%%%%%%%%%%%

\subsection{Proof of the main theorem}
We first note that there exists $s \in \mathbb{F}_q$ and $t \in \mathbb{F}_q^*$ such that $\alpha^2+s\alpha+t=0$, or equivalently, $\alpha^2=s\alpha+t$.

\begin{lem}
	The property $\Phi\left(\agl(2,q^2),H_q \right)$ does not hold.
\end{lem}

\begin{proof}
	Let $g= \left( \begin{pmatrix} 0 \\ \alpha +s \end{pmatrix}, \begin{pmatrix} 1 & \alpha \\ 0 & 1 \end{pmatrix} \right) \in \agl(2,q^2)$. We have
	$$
	g^2 =\left( \begin{pmatrix} \alpha^2+s \alpha \\ 0 \end{pmatrix}, \mathrm{I}_2 \right)=\left( \begin{pmatrix} s \alpha +t+s \alpha \\ 0 \end{pmatrix}, \mathrm{I}_2 \right)=\left( \begin{pmatrix} t \\ 0 \end{pmatrix}, \mathrm{I}_2 \right) \in H_q.
	$$
	Let $h= \left( \begin{pmatrix} u \\ v \end{pmatrix}, \mathrm{I}_2 \right) \in H_q$ with $u,v \in \mathbb{F}_q$.  As $t \neq 0$, we have
	\begin{align*}
	(gh)^2&=  \left [\left ( \begin{pmatrix}0 \\ \alpha +s \end{pmatrix}, \begin{pmatrix} 1 & \alpha \\ 0 & 1 \end{pmatrix} \right) \left ( \begin{pmatrix} u \\ v \end{pmatrix}, \begin{pmatrix} 1 & 0 \\ 0 & 1 \end{pmatrix} \right) \right]^2\\
	&= \left ( \begin{pmatrix} u + v \alpha \\ \alpha + s+ v \end{pmatrix}, \begin{pmatrix} 1 & \alpha \\ 0 & 1 \end{pmatrix} \right)^2\\
	&= \left ( \begin{pmatrix} v \alpha +t \\ 0 \end{pmatrix}, \begin{pmatrix}1 & 0 \\ 0 & 1 \end{pmatrix} \right) \neq (0, \mathrm{I}_2).
	\end{align*}
	Consequently, $\Phi\left(\agl(2,q^2),H_q \right)$ does not hold.
\end{proof}

\begin{lem}
	The normalizer of $H_q$ in $\agl(2,q^2)$ is given by:
	$$
	N_{\agl(2,q^2)}(H_q)= \left \lbrace (a,A) \mid a \in \mathbb{F}_{q^2}^2, A \in \gl_2(\mathbb{F}_q)  \right \rbrace.
	$$\label{lem:normalizer}
\end{lem}
\begin{proof}
	For any $g=(a,A) \in \agl(2,q^2)$ and $h=(b, \mathrm{I}_2) \in H_q$, we have
	\begin{equation}\label{equ:12ga}
	ghg^{-1} =(a,A)(b,\mathrm{I}_2)(a,A)^{-1}=(Ab, \mathrm{I}_2).
	\end{equation}
	Let $g = (a,A) \in N_{\agl(2,q^2)}(H_q)$, where $A = \begin{pmatrix} u & v \\ w & z  \end{pmatrix}$. Since $g\in N_{\agl(2,q^2)}(H_q)$, we know that $g (b,\mathrm{I}_2) g^{-1} = (Ab,\mathrm{I}_2) \in H_q$, for $(b,I_2) \in H_q$ (see \eqref{equ:12ga}). In particular, 
	\begin{align*}
	Ab \in \begin{pmatrix} \mathbb{F}_q \\ \mathbb{F}_q \end{pmatrix}, \mbox{ for } b \in \left\{\begin{pmatrix} 1 \\ 0 \end{pmatrix},\begin{pmatrix} 0 \\ 1 \end{pmatrix} \right\}.
	\end{align*}
	Therefore, the columns of $A$ are elements of $\begin{pmatrix} \mathbb{F}_q \\ \mathbb{F}_q \end{pmatrix}$ and so $A \in \gl_2(\mathbb{F}_q)$. In other words, $N_{\agl(2,q^2)}(H_q)\subseteq  \left \lbrace (a,A) \mid a \in \mathbb{F}_{q^2}^2, A \in \gl_2(\mathbb{F}_q)  \right \rbrace$.
	
	Conversely, if $A \in \gl_2(\mathbb{F}_q)$ and $a\in \mathbb{F}_{q^2}^2$ then, by \eqref{equ:12ga}, it is easy to see that $(a,A) \in N_{\agl(2,q^2)}(H)$. This completes the proof.
\end{proof}

An immediate consequence of Lemma~\ref{lem:normalizer} is that $H_q$ is not a normal subgroup of $\agl(2,q^2)$.

\begin{thm}
	The subgroup $H_q$ of $\agl(2,q^2)$ is a perfect code.
\end{thm}

\begin{proof}
	Since $H_q$ is a $2$-group, we may apply Lemma~\ref{lem:char2}. Consequently, we only need to show that $\Phi\left(N_{\agl(2,q^2)}(H_q),H_q\right)$ holds. Let $g=(a,A) \in N_{\agl(2,q^2)}(H_q)$ such that $(a,A)^2 = ((A+\mathrm{I}_2)a,A^2) \in H_q$, that is, $(A+\mathrm{I}_2)a \in \begin{pmatrix} \mathbb{F}_q \\ \mathbb{F}_q \end{pmatrix}$ and $A^2=\mathrm{I}_2$. Let us prove that there exists $b \in \begin{pmatrix} \mathbb{F}_q \\ \mathbb{F}_q \end{pmatrix}$ such that $\left((a,A)(b,\mathrm{I}_2)\right)^2 = (0,\mathrm{I}_2) $. 
	
	First we note that if $A = \mathrm{I}_2$, then $(a,A)^2 = (a,\mathrm{I}_2)^2 = (a+a,\mathrm{I}_2) = (0,\mathrm{I}_2)$. Thus, for $b = 0$, we have $\left((a,A)(b,\mathrm{I}_2)\right)^2 = (0,\mathrm{I}_2) $. Therefore, we assume henceforth that $A\neq \mathrm{I}_2$. 
	
	Suppose that $(A+\mathrm{I}_2)a= \begin{pmatrix} a_1 \\ a_2 \end{pmatrix} \in \begin{pmatrix} \mathbb{F}_q \\ \mathbb{F}_q \end{pmatrix}$. Recall that the spectrum of a square matrix is the multiset consisting of all its eigenvalues. Since $A^2 = \mathrm{I}_2$, the spectrum of $A$ is the multiset $\{1 ,1\}$ and $\mathrm{tr}(A)=0$. Thus, we may write 
	$
	A=\begin{pmatrix} t & v \\ u & t \end{pmatrix}, 
	$
	for some $t,u,v \in \mathbb{F}_{q}$.	
	As $\mathrm{det}(A)=t^2+uv=1$, we obtain the equality $(t-1)^2=uv$.
	
	For any $h=(b, \mathrm{I}_2) \in H_q$, we have
	\begin{equation}\label{eq:ghsquare}
	(gh)^2=\left( (a,A)(b,\mathrm{I}_2) \right)^2=\left( a+Ab,A \right)^2=\left( (A+\mathrm{I}_2)(a+b), \mathrm{I}_2 \right).
	\end{equation}
	
	\noindent{\it 1. Assume that $t=1$.} 
	
	In this case, $uv=(t-1)^2=0$ and so $u=0$ or $v=0$. Without loss of generality, assume that $u = 0$. Then, $A = \begin{pmatrix} 1 & v \\ 0 & 1 \end{pmatrix}$ and since $A \neq \mathrm{I}_2$, we must have $v\neq 0$. Hence, $A+\mathrm{I}_2 =  \begin{pmatrix} 0 & v \\ 0 & 0 \end{pmatrix}$. Consequently, $(A+\mathrm{I}_2)a= \begin{pmatrix} a_1 \\ a_2 \end{pmatrix}$ implies that $a_2 = 0$. Taking $b= \begin{pmatrix} 0 \\ -v^{-1}a_1 \end{pmatrix} \in \mathbb{F}_q^2$  in \eqref{eq:ghsquare}, we have $(gh)^2= (0,\mathrm{I}_2)$.
	\\ \\
	\noindent{\it 2. Assume that $t \neq 1$.} 
	
	Let $t'=t-1$. Since $uv=(t-1)^2=t'^2 \neq 0$, we know that $u \neq 0$, $v \neq 0$, and $A+I= \begin{pmatrix} 
	t' & u^{-1}t'^2 \\ u & t'
	\end{pmatrix}$. Note that $(A+\mathrm{I}_2)a= \begin{pmatrix} a_1 \\ a_2 \end{pmatrix}$ implies $a_1=u^{-1}t'a_2$. By letting $b=\begin{pmatrix} 0 \\ -t'^{-1}a_2 \end{pmatrix} \in \mathbb{F}_q^2$ in \eqref{eq:ghsquare}, we have $(gh)^2=(0,\mathrm{I}_2)$.
	
	We conclude that $\Phi\left( N_{\agl(2,q^2)(H_q)},H_q \right)$ holds, therefore $H_q$ is a subgroup perfect code of $\agl(2,q^2)$.
\end{proof}

\subsubsection*{Acknowledgment} We are grateful to the anonymous reviewers for their insightful comments, which helped improve this paper.

\bibliographystyle{plain}

\end{document}